\documentclass[final,leqno]{siamltex704}
\usepackage{amsmath}
\usepackage{amssymb}
\usepackage{graphicx}
\usepackage[notcite,notref]{showkeys}
\usepackage{mathrsfs}
\usepackage{float}
\usepackage{hyperref}

\newcommand{\bn}{{\bf n}}

\newcommand{\pT}{{\partial T}}
\def\bbQ{\mathbb{Q}}
\def\T{{\mathcal T}}
\def\E{{\mathcal E}}

\def\W{{\mathcal W}}
\def\l{{\langle}}
\def\r{{\rangle}}
\def\3bar{{|\hspace{-.02in}|\hspace{-.02in}|}}
\newtheorem{defi}{Definition}[section]

\newtheorem{algorithm}{Algorithm}
\newtheorem{remark}[theorem]{Remark}

\title{ Effective Implementation of the Weak Galerkin Finite Element
Methods for the Biharmonic Equation}

\author{Lin Mu\thanks{Computer Science and Mathematics Division
Oak Ridge National Laboratory, Oak Ridge, TN, 37831,USA
(mul1@ornl.gov). This research was supported in part by the U.S.~Department of Energy, Office of Science, Office of Advanced Scientific Computing Research, Applied Mathematics program under award number ERKJE45; and by the Laboratory Directed Research and Development program at the Oak Ridge National Laboratory, which is operated by UT-Battelle, LLC., for the U.S.~Department of Energy under Contract DE-AC05-00OR22725.}
\and Junping
Wang\thanks{Division of Mathematical Sciences, National Science
Foundation, Arlington, VA 22230 (jwang@\break nsf.gov). The research
of Wang was supported by the NSF IR/D program, while working at
National Science Foundation. However, any opinion, finding, and
conclusions or recommendations expressed in this material are those
of the author and do not necessarily reflect the views of the
National Science Foundation,} \and Xiu Ye\thanks{Department of
Mathematics, University of Arkansas at Little Rock, Little Rock, AR
72204 (xxye@ualr.edu). This research was supported in part by
National Science Foundation Grant DMS-1115097.}}

\begin{document}

\maketitle

\begin{abstract}

The weak Galerkin (WG) methods have been introduced in \cite{mwy,zz} for
solving the biharmonic equation. The purpose of this paper is to develop an algorithm to implement the WG methods effectively. This can be achieved by eliminating local unknowns to obtain a global system with significant reduction of size. In fact this reduced global system  is equivalent to the Schur complements of the WG methods. The unknowns of the Schur complement of the WG method are those defined on the element boundaries.  The equivalence of the WG method and its Schur complement is established. The numerical results demonstrate the effectiveness of this new implementation technique.
\end{abstract}

\begin{keywords}
weak Galerkin, finite element methods, weak Laplacian,
biharmonic equations, polyhedral meshes
\end{keywords}

\begin{AMS}
Primary, 65N15, 65N30, 76D07; Secondary, 35B45, 35J50
\end{AMS}
\pagestyle{myheadings}

\section{Introduction}

We consider the biharmonic equation of the form
\begin{eqnarray}
\Delta^2 u&=&f,\quad \mbox{in}\;\Omega,\label{pde}\\
u&=&g,\quad\mbox{on}\;\partial\Omega,\label{bc-d}\\
\frac{\partial u}{\partial
n}&=&g_n\quad\mbox{on}\;\partial\Omega.\label{bc-n}
\end{eqnarray}
\medskip
For the biharmonic problem (\ref{pde}) with Dirichlet and Neumann boundary conditions (\ref{bc-d}) and (\ref{bc-n}), the corresponding
variational form is given by seeking $u\in H^2(\Omega)$ satisfying
$u|_{\partial \Omega}=g$ and $\frac{\partial u}{\partial
n}|_{\partial \Omega}=\phi$ such that
\begin{equation}\label{wf}
(\Delta u, \Delta v) = (f, v),\qquad \forall v\in H_0^2(\Omega),
\end{equation}
where $H_0^2(\Omega)$ is the subspace of $H^2(\Omega)$ consisting of
functions with vanishing value and normal derivative on
$\partial\Omega$.

Conforming finite element methods for this fourth order equation
require  finite element spaces to be subspaces of $H^2(\Omega)$ or $C^1(\Omega)$.
Due to the complexity of construction of
$C^1$  elements, $H^2$ conforming methods are rarely used in
practice for solving the biharmonic equation. Due to this reason, many  nonconforming or discontinuous finite element methods have been developed for solving the biharmonic equation.
Morley element \cite{morley} is a well known nonconforming element
for the biharmonic equation for its simplicity. $C^0$ interior
penalty methods were studied in \cite{bs, ghlmt}. In \cite{ms}, a $hp$-version interior
penalty discontinuous Galerkin (DG) methods  were developed for the
biharmonic equation.

Weak Galerkin methods refer to general finite element techniques for partial differential equations and were  first introduced  in \cite{wy} for second order elliptic equations. They are by designing using discontinuous approximating functions on general meshes to avoid construction of complicated elements such as $C^1$ conforming elements.
In general, weak Galerkin finite element formulation can be derived directly from the variational form of the PDE by replacing the corresponding derivatives by the weak derivatives and adding a parameter independent stabilizer. Obviously, the WG method for the biharmonic equation should have the form
\begin{equation}\label{wf1}
(\Delta_w u_h, \Delta_w v)+s(u_h,v) = (f, v),
\end{equation}
where $s(\cdot,\cdot)$ is a parameter independent stabilizer. The WG formulation (\ref{wf1}) in its primary form is symmetric and positive definite.

The main idea of weak Galerkin finite
element methods is the use of weak functions and their corresponding weak derivatives in algorithm design.
For the biharmonic equations, weak function has the form  $v=\{v_0,v_b,v_n\}$  with $v=v_0$ inside of each element and $v=v_b$, $\nabla v\cdot\bn=v_n$ on the boundary of
the element.  In the weak Galerkin method introduced in \cite{mwy}, $v_0$ and $v_b$ are approximated by $k$th order polynomials and $v_n$ is approximated by the polynomial of order $k-1$. This method has been improved in \cite{zz} through polynomial order reduction where
$v_b$ and $v_n$ are both approximated by the polynomials of degree $k-1$.

Introductions of weak functions and weak derivatives make
the WG methods highly flexible. It also creates additional degrees of freedom associated with $v_b$ and $v_n$. The purpose of this paper is to develop an algorithm to implement the WG methods introduced in \cite{mwy,zz} effectively. This can be achieved by deriving the Schur complements of the WG methods and  eliminating the unknown $u_0$ from the globally coupled systems. Variables $u_b$ and $u_n$ defined on the element boundaries are the only unknowns of the Schur complements which significantly reduce globally coupled unknowns. We prove that the reduced system is symmetric and positive definite. The equivalence of the WG method and its Schur complement is also established. The results of this paper is based on the weak Galerkin method developed in \cite{zz}. The theory can also be applied to the WG method introduced in \cite{mwy} directly.

The paper is organized as follows. A weak Laplacian operator is introduced in Section \ref{Section:weak-Laplacian}. In Section \ref{Section:wg-fem}, we provide a  description for the  WG finite element scheme for the biharmonic equation introduced in \cite{zz}.
In Section \ref{Section:Schur}, a Schur complement formulation of the WG method is derived to reduce the cost in the implementation.
Numerical experiments are conducted in Section \ref{Section:numerical-results}.

\section{Weak Laplacian and discrete weak Laplacian}\label{Section:weak-Laplacian}

Let $T$ be any polygonal or polyhedral domain with boundary
$\partial T$. A {\em weak function} on the region $T$ refers to a
function $v=\{v_0, v_b, v_{n}\}$ such that $v_0\in L^2(T)$,
$v_b\in H^{\frac12}(\partial T)$, and $v_n\in
H^{-\frac12}(\partial T)$.  The first component $v_0$ can be understood as the value of $v$
in $T$ and the second and the third components $v_b$ and $v_n$ represent $v$ on $\pT$ and $\nabla v\cdot\bn$ on $\pT$,   where $\bn$ is the outward normal
direction of $T$ on its boundary. Note that $v_b$ and $v_n$ may not necessarily be related to
the trace of $v_0$  and $\nabla v_0\cdot\bn$ on $\partial K$ should  traces be well-defined.

Denote by $\W(T)$
the space of all weak functions on $T$; i.e.,
\begin{equation}\label{www}
\W(T) = \{v=\{v_0, v_b, v_n \}:\ v_0\in L^2(T),\; v_b\in
H^{\frac12}(\partial T), \ v_n\in H^{-\frac12}(\partial
T)\}.
\end{equation}
Let  $(\cdot,\cdot)_T$ stand for the $L^2$-inner product in
$L^2(T)$, $\langle\cdot,\cdot\rangle_\pT$ be the inner product in
$L^2(\pT)$. For convenience, define $G^2(T)$  as follows
$$
G^2(T)=\{\varphi: \ \varphi\in H^1(T), \Delta\varphi\in L^2(T)\}.
$$
It is clear that, for any $\varphi\in G^2(T)$, we have
$\nabla\varphi\in H(div,T)$. It follows that
$\nabla\varphi\cdot\bn\in  H^{-\frac12}(\partial T)$ for any
$\varphi\in G^2(T)$.

\medskip

\begin{defi}(Weak Laplacian)
The dual of $L^2(T)$ can be identified with itself by using the
standard $L^2$ inner product as the action of linear functionals.
With a similar interpretation, for any $v\in \W(T)$, the {\em weak
Laplacian} of $v=\{v_0, v_b,v_n \}$ is defined as a linear
functional $\Delta_w v$ in the dual space of $G^2(T)$ whose action
on each $\varphi\in G^2(T)$ is given by
\begin{equation}\label{wl}
(\Delta_w v, \ \varphi)_T = (v_0, \ \Delta\varphi)_T-\l v_b,\ \nabla\varphi\cdot\bn\r_\pT +\l
v_n, \ \varphi\r_\pT,
\end{equation}
where $\bn$ is the outward normal direction to $\partial T$.
\end{defi}

The Sobolev space $H^2(T)$ can be embedded into the space $\W(T)$ by
an inclusion map $i_\W: \ H^2(T)\to \W(T)$ defined as follows
$$
i_\W(\phi) = \{\phi|_{T}, \phi|_{\partial T}, \nabla\phi|\cdot\bn_{\partial T}\},\qquad \phi\in H^2(T).
$$
With the help of the inclusion map $i_\W$, the Sobolev space $H^2(T)$
can be viewed as a subspace of $\W(T)$ by identifying each $\phi\in
H^2(T)$ with $i_\W(\phi)$. Analogously, a weak function
$v=\{v_0,v_b,v_n\}\in \W(T)$ is said to be in $H^2(T)$ if it can be
identified with a function $\phi\in H^2(T)$ through the above
inclusion map. It is not hard to see that the weak Laplacian is
identical with the strong Laplacian (i.e., $\Delta_w v=\Delta v$) for
smooth functions $v\in H^2(T)$.

Next, we introduce a discrete weak Laplacian operator by
approximating $\Delta_w$ in a polynomial subspace of the dual of
$G^2(T)$. To this end, for any non-negative integer $r\ge 0$, denote
by $P_{r}(T)$ the set of polynomials on $T$ with degree no more than
$r$.

\begin{defi} (Discrete Weak Laplacian)
A discrete weak Laplacian operator, denoted by $\Delta_{w,r,T}$,
is defined as the unique polynomial $\Delta_{w,r,T}v \in P_r(T)$ that
satisfies the following equation
\begin{equation}\label{dwl}
(\Delta_{w,r,T} v, \ \varphi)_T = ( v_0, \ \Delta\varphi)_T-\l v_b,\
\nabla\varphi\cdot\bn\r_\pT +\l v_n, \ \varphi\r_\pT,\quad
\forall \varphi\in P_r(T).
\end{equation}
\end{defi}

\section{Weak Galerkin Finite Element Methods}\label{Section:wg-fem}

Let ${\cal T}_h$ be a
partition of the domain $\Omega$ consisting of polygons in 2D or
polyhedra in 3D. Assume that ${\cal T}_h$ is shape regular in the
sense that a set of conditions defined in \cite{wy-mixed} are
satisfied. Denote by ${\cal E}_h$ the set of all edges or flat faces in ${\cal
T}_h$, and let ${\cal E}_h^0={\cal E}_h\backslash\partial\Omega$ be
the set of all interior edges or flat faces.

Since $v_n$ represents $\nabla v\cdot\bn$, obviously, $v_n$  is dependent on  $\bn$.
To ensure $v_n$ a single value function on $e\in\E_h$,
we introduce a set of normal directions on ${\cal E}_h$ as follows
\begin{equation}\label{thanks.101}
{\cal D}_h = \{\bn_e: \mbox{ $\bn_e$ is unit and normal to $e$},\
e\in {\cal E}_h \}.
\end{equation}
For a given integer $k\ge 2$, let $V_h$ be the weak Galerkin finite
element space associated with $\T_h$ defined as follows
\[
V_h=\{v=\{v_0,v_b, v_{n}\}:\ v_0|_T\in P_{k}(T),  v_b|_e\in
P_{k-1}(e),
 v_{n}|_e\in P_{k-1}(e), e\subset\partial T\},
\]
where $v_n$ can be viewed as an approximation of $\nabla v\cdot\bn_e$.
Denote by $V_h^0$ a subspace of $V_h$ with vanishing traces; i.e.,
$$
V_h^0=\{v=\{v_0,v_b,v_{n}\}\in V_h, {v_b}|_e=0,\ {v_{n}}|_e=0,\
e\subset\partial T\cap\partial\Omega\}.
$$
Denote by $\Delta_{w,k-2}$ the discrete weak Laplacian operator on
the finite element space $V_h$ computed by using
(\ref{dwl}) on each element $T$ for $k\ge 2$; i.e.,
$$
(\Delta_{w,k-2}v)|_T =\Delta_{w,k-2, T} (v|_T),\qquad \forall v\in
V_h.
$$
For simplicity of notation, from now on we shall drop the subscript
$k-2$ in the notation $\Delta_{w,k-2}$ for the discrete weak Laplacian.

For each element $T\in \T_h$, denote by $Q_0$ the $L^2$ projection
from $L^2(T)$ to $P_k(T)$ and by $Q_b$ the $L^2$ projection from
$L^2(e)$ to $P_{k-1}(e)$. Denote by $\bbQ_h$ the $L^2$ projection
onto the local discrete gradient space $P_{k-2}(T)$. Now for any $u\in H^2(\Omega)$, we
can define a projection into the finite element space $V_h$ such
that on the element $T$, we have
$$
Q_h u = \{Q_0 u,\;Q_bu,\; Q_b (\nabla u\cdot\bn_e)\}.
$$
We also introduce the following notation
$$
(\Delta_w v,\;\Delta_w w)_h=\sum_{T\in {\cal T}_h}(\Delta_w
v,\;\Delta_w w)_T.
$$
For any $u_h=\{u_0,u_b,u_{n}\}$ and
$v=\{v_0,v_b,v_{n}\}$ in $V_h$, we introduce a stabilizer as
follows
\[
s(u_h, v)=\sum_{T\in\T_h} h_T^{-1}\langle \nabla u_0\cdot\bn_e-u_{n}, \
\nabla v_0\cdot\bn_e-v_{n}\rangle_\pT + \sum_{T\in\T_h} h_T^{-3}\langle
Q_b u_0-u_{b},Q_b v_0-v_{b}\rangle_\pT.
\]
In the definition above, the first term is to enforce the connections between normal derivatives of $u_0$ along $\bn_e$ and its approximation $u_n$. Now we can define the bilinear form for the weak Galerkin formulation,
\begin{equation}\label{a-form}
a(v,w)=(\Delta_w v,\ \Delta_w w)_h+s(v,\ w).
\end{equation}

\begin{algorithm}\label{algorithm1} (WG method)
A numerical approximation for (\ref{pde})-(\ref{bc-n}) can be
obtained by seeking $u_h=\{u_0,\;u_b,\ u_{n}\}\in V_h$
satisfying $u_b=Q_b g$ and $u_{n}=Q_bg_n(\bn\cdot\bn_e)$ on $\partial \Omega$
and the following equation:
\begin{equation}\label{wg}
a(u_h,\ v)=(f,\;v_0), \quad\forall\
v=\{v_0,\; v_b,\ v_{n}\}\in V_h^0.
\end{equation}
\end{algorithm}

Define a mesh-dependent semi norm $\3bar\cdot\3bar$ in the finite element
space $V_h$ as follows
\begin{equation}\label{3barnorm}
\3bar v\3bar^2=a(v,\ v)=(\Delta_wv,\ \Delta_w v)_h+s(v,\;v),\qquad v\in V_h.
\end{equation}

It has been proved in \cite{zz} that $\3bar\cdot\3bar$ is a norm in $V_h^0$ and therefore the weak Galerkin Algorithm \ref{algorithm1} has a unique solution.

\begin{theorem} Let $u_h\in V_h$  be the weak Galerkin finite element solution arising from
(\ref{wg}) with finite element functions of order $k\ge 2$. Assume
that the exact solution of (\ref{pde})-(\ref{bc-n} ) is sufficient
regular such that $u\in H^{\max\{k+1,4\}}(\Omega)$. Then, there exists a
constant $C$ such that
\begin{equation}\label{err1}
\3bar u_h-Q_hu\3bar \le Ch^{k-1}
\left(\|u\|_{k+1}+h\delta_{k-2,0}\|u\|_{4}\right).
\end{equation}
\end{theorem}

\begin{theorem}
Let $u_h\in V_h$ be the weak Galerkin finite element solution
arising from (\ref{wg}) with finite element functions of order
$k\ge 3$. Assume that the exact solution of
(\ref{pde})-(\ref{bc-n} ) is sufficient regular such that $u\in
H^{k+1}(\Omega)$ with $k\ge 3$. Then, there exists a constant $C$ such that
\begin{equation}\label{err2}
\|Q_0u-u_0\| \le Ch^{k+1}\|u\|_{k+1}.
\end{equation}
\end{theorem}

\section{The Schur Complement of the WG Method}\label{Section:Schur}

To reduce the number of globally coupled unknowns of the WG method  (\ref{wg}), its Schur complement will be derived by eliminating $u_0$.
To start the local elimination procedure, denote
by $V_h(T)$ the restriction of $V_h$ on $T$, i.e.
\[
V_h(T)=\{v=\{v_0,v_b,v_n\}\in V_h^0, v({\bf x})= 0,\; \mbox{for}\; {\bf x}\notin T\}.
\]

\begin{algorithm}\label{algorithm2}(The Schur Complement of the WG Method)
An approximation for
(\ref{pde})-(\ref{bc-n}) is given by seeking $w_h=\{w_0,\;w_b,\ w_{n}\}\in V_h$
satisfying $w_b=Q_b g$ and $w_{n}=Q_bg_n(\bn\cdot\bn_e)$ on $\partial \Omega$
and a global equation
\begin{equation}\label{wg1}
a(w_h,\ v)=0, \quad\forall\
v=\{0,\; v_b,\ v_{n}\}\in V_h^0,
\end{equation}
and a local system on each element $T\in \T_h$,
\begin{eqnarray}
a(w_h,\ v)=(f,\;v_0), \quad\forall\ v=\{v_0,\; 0,\ 0\}\in V_h(T).\label{wg2}
\end{eqnarray}
\end{algorithm}

\begin{remark}
Algorithm \ref{algorithm2} consists two parts: a local system (\ref{wg2}) solved on each element $T\in\T_h$ for eliminating $w_0$ and a global system (\ref{wg1}). The global system (\ref{wg1}) has $w_b$ and $w_n$ as its only unknowns that will reduces the number of the unknowns from the WG system (\ref{wg}) by half.
\end{remark}

\begin{theorem}\label{thm-h}
Let $w_h=\{w_0, w_b,w_n\}\in V_h$  and $u_h=\{u_0,u_b,u_n\}\in V_h$  be the solutions of Algorithm \ref{algorithm2} and Algorithm \ref{algorithm1} respectively. Then we have
\begin{equation}\label{main}
w_h=u_h.
\end{equation}
\end{theorem}
\begin{proof}
For any $v=\{v_0,v_b,v_n\}\in V_h^0$, we have $v=\{v_0,0,0\}+\{0,v_b,v_n\}$. Therefore it is easy to see that $w_h$ is a solution of WG method (\ref{wg}). The uniqueness of the WG method proved in \cite{zz} implies $w_h=u_h$ which proved the theorem.
\end{proof}

For given $w_b$, $w_n$ and $f$, let $w_h$ be the unique solution of the local system
(\ref{wg2}) which is  a function of $w_b$, $w_n$ and $f$,
\begin{equation}\label{ttt}
w=w(w_b,w_n,f)=\{w_{0}(w_b,w_n,f),w_b,w_n\}\in V_h(T).
\end{equation}
Superposition implies
\begin{equation}\label{sp}
w(w_b,w_n,f)=w(w_b,w_n,0)+w(0,0,f).
\end{equation}

Using the equation (\ref{sp}), (\ref{wg1}) becomes
\begin{equation}\label{wg3}
a(w(w_b,w_n,0),v)=-a(w(0,0,f),v),\quad\forall v=\{0,v_b,v_n\}\in V_h^0.
\end{equation}

\smallskip

\begin{lemma}
System (\ref{wg1})  is symmetric and positive definite.
\end{lemma}

\smallskip

\begin{proof}
Since the system (\ref{wg1}) is equivalent to (\ref{wg3}), we will prove that the system (\ref{wg3}) is symmetric and positive definite.
It follows from the definition of $w(w_b,w_n,0)$ and (\ref{wg2})  that
\begin{equation}\label{ttt1}
a(w(w_b,w_n,0),v)=0,\quad\forall v=\{v_0,0,0\}\in V_h(T).
\end{equation}

Combining the equation (\ref{ttt}) with $f=0$ and (\ref{ttt1}) implies that for all $v=\{0,v_b,v_n\}\in V_h^0$
\[
a(w(w_b,w_n,0),v)= a(w(w_b,w_n,0),v(v_b,v_n,0)),
\]
which implies that the system (\ref{wg3}) is symmetric. Next we will prove that $v_b=v_n=0$ for any $v\in V_h^0$ if
$$a(v(v_b,v_n,0),v(v_b,v_n,0))=0.$$
It has been proved in \cite{zz,zzw} that $\3bar v \3bar^2=a(v,v)=0$ implies $v=0$ if $v\in V_h^0$. Thus we have $v=0$.
 The uniqueness of the system (\ref{wg2}) implies $v_b=v_n=0$.  We have proved the lemma.
\end{proof}

\section{Numerical Experiments}\label{Section:numerical-results}
This section shall present several numerical experiments to illustrate the HWG algorithm devised in this article.

The numerical experiments are conducted in the weak Galerkin finite element space:
$$V_h=\{v=\{v_0,v_b,v_{n}\bn_e\}, \ v_0\in P_{k}(T),\ v_b\in P_{k-1}(e),\ v_{n}\in P_{k-1}(e), T\in\mathcal{T}_h, e\in\mathcal{E}_h\}.$$
For any given $v=\{v_0,v_b,v_{n}\bn_e\}\in V_h$, its discrete weak Laplacian, $\Delta_wv\in P_{k-2}(T)$, is computed locally by the following equation
$$(\Delta_w v,\psi)_T=(v_0,\Delta\psi)_T+\langle v_{n}{\bf n}_e\cdot{\bf n},\psi\rangle_{\partial T}-\langle v_b,\nabla\psi\cdot{\bf n}\rangle_{\partial T},\ \forall\psi\in P_{k-2}(T).$$
The error for the WG solution will be measured in two norms defined as follows:
\begin{eqnarray}
|||v_h|||^2:&=&\sum_{T\in\mathcal{T}_h}\bigg(\int_T|\Delta_w v_h|^2dx+h_T^{-1}\int_{\partial T}|(\nabla v_0-v_{n}{\bf n}_e)\cdot{\bf n}|^2ds \\
&&+h_T^{-3}\int_{\partial T}(Q_bv_0-v_b)^2ds\bigg),\qquad\qquad\  (\mbox{A discrete $H^2$-norm}),\notag\\
\|v_h\|^2:&=&\sum_{T\in\mathcal{T}_h}\int_T|v_0|^2dx,\qquad\qquad\qquad\qquad  (\mbox{Element-based $L^2$-norm}).
\end{eqnarray}

In the following setting, we will choose $k=2$ and $k=3$ for testing. 

\begin{table}[h]
\caption{Example 1. Convergence rate with $k=2$.}\label{ex1_k2}
\center
\begin{tabular}{||c||cccc||}
\hline\hline
$h$ & $\3bar u_h-Q_h u\3bar$ & order & $\|u_0-Q_0 u\|$ & order \\
\hline\hline
   1/4    &2.4942e-01 &            &3.3400e-02 &        \\ \hline
   1/8    &1.3440e-01 &8.9202e-01  &9.1244e-03 &1.8720  \\ \hline
   1/16   &7.2244e-02 &8.9562e-01  &2.6093e-03 &1.8061  \\ \hline
   1/32   &3.8252e-02 &9.1734e-01  &7.3363e-04 &1.8305  \\ \hline
   1/64   &1.9681e-02 &9.5877e-01  &1.9488e-04 &1.9125  \\ \hline
   1/128  &9.9257e-03 &9.8753e-01  &4.6501e-05 &2.0673  \\ \hline
\hline
\end{tabular}
\end{table}

\begin{table}[h]
\caption{Example 1. Convergence rate with $k=3$.}\label{ex1_k3}
\center
\begin{tabular}{||c||cccc||}
\hline\hline
$h$ & $|||u_h-Q_h u|||$ & order & $\|u_0-Q_0 u\|$ & order \\
\hline\hline
   1/4    &6.2092e-02 &        &4.9565e-03 &        \\ \hline
   1/8    &2.2944e-02 &1.4363  &4.6283e-04 &3.4208  \\ \hline
   1/16   &6.8389e-03 &1.7463  &3.7550e-05 &3.6236  \\ \hline
   1/32   &1.7486e-03 &1.9676  &2.4198e-06 &3.9559  \\ \hline
   1/64   &4.3878e-04 &1.9946  &1.5181e-07 &3.9946  \\ \hline
   1/128  &1.0983e-04 &1.9982  &8.9374e-09 &4.0862  \\ \hline
\hline
\end{tabular}
\end{table}

\subsection{Example 1}
Consider the forth order problem that seeks an unknown function $u=u(x,y)$ satisfying
$$
-\Delta^2 u=f
$$
in the square domain $\Omega=(0,1)^2$ with homogeneous Dirichlet
boundary condition. The exact solution is given by $u(x,y)=x^2(1-x)^2y^2(1-y^2)$, and the function $f=f(x,y)$ is given to match the
exact solution.

The HWG algorithm is performed on a sequence of uniform triangular meshes. The mesh is constructed as follows: 1) partition the domain $\Omega$ into $n\times n$ sub-rectangles; 2) divide each square element into two triangles by the diagonal line with a negative slope. We denote the mesh size as $h=1/n.$

Table \ref{ex1_k2} present the error profiles with the mesh size $h$ for $k=2.$ Here, it is observed that $\3bar u_h-Q_h u\3bar$ converges to zero at the optimal rate $O(h)$ as the mesh is refined. The third column in Table \ref{ex1_k2} shows the convergence rate of $\|u_0-Q_0u\|$ is at sub-optimal rate $O(h^2).$ Secondly, in Table \ref{ex1_k3} we investigate the same problem for $k=3.$ It shows that the $\3bar u_h-Q_hu\3bar$ and $\|u_h-Q_0u\|$ are converged at the rate of $O(h^2)$ and $O(h^4)$, which validate the theoretical conclusion in (\ref{err1})-(\ref{err2}).

\begin{table}[h]
\caption{Example 2. Convergence rate with $k=2$.}\label{ex2_k2}
\center
\begin{tabular}{||c||cccc||}
\hline\hline
$h$ & $\3bar u_h-Q_h u\3bar$ & order & $\|u_0-Q_0 u\|$ & order \\
\hline\hline
   1/4    &1.1977e+01 &            &1.5977     &        \\ \hline
   1/8    &6.3606     &9.1305e-01  &4.2748e-01 &1.9020  \\ \hline
   1/16   &3.3570     &9.2199e-01  &1.1740e-01 &1.8644  \\ \hline
   1/32   &1.7395     &9.4854e-01  &3.1336e-02 &1.9056  \\ \hline
   1/64   &8.8243e-01 &9.7910e-01  &8.0433e-03 &1.9620  \\ \hline
   1/128  &4.4185e-01 &9.9793e-01  &2.0110e-03 &1.9999  \\ \hline
\hline
\end{tabular}
\end{table}

\begin{table}[h]
\caption{Example 2. Convergence rate with $k=3$.}\label{ex2_k3}
\center
\begin{tabular}{||c||cccc||}
\hline\hline
$h$ & $\3bar u_h-Q_h u\3bar$ & order & $\|u_0-Q_0 u\|$ & order \\
\hline\hline
   1/4    &3.9757     &        &3.7061e-01 &        \\ \hline
   1/8    &1.2465     &1.6734  &3.0620e-02 &3.5973  \\ \hline
   1/16   &3.5336e-01 &1.8186  &2.2781e-03 &3.7486  \\ \hline
   1/32   &9.1275e-02 &1.9528  &1.4426e-04 &3.9811  \\ \hline
   1/64   &2.3058e-02 &1.9849  &8.9582e-06 &4.0093  \\ \hline
   1/128  &5.7870e-03 &1.9944  &5.5593e-07 &4.0102  \\ \hline
\hline
\end{tabular}
\end{table}

\subsection{Example 2} Let $\Omega=(0,1)^2$ and exact solution $u(x,y)=\sin(\pi x)\sin(\pi y)$. The boundary conditions $g$, $g_n$, and $f$ are given to match the exact solution.

Similarly, the uniform triangular mesh is used for testing. Table \ref{ex2_k2}-Table \ref{ex2_k3} present the error for $k=2$ and $k=3$ respectively. We can observe the convergence rates measured in $\3bar u_h-Q_h u\3bar$ and $\|u_0-Q_0u\|$ are $O(h)$, $O(h^2)$ for $k=2$, and $O(h^2)$, $O(h^4)$ for $k=3$.

\begin{table}[h]
\caption{Example 3. Convergence rate with $k=2$.}\label{ex3_k2}
\center
\begin{tabular}{||c||cccc||}
\hline\hline
Mesh & $\3bar u_h-Q_h u\3bar$ & order & $\|u_0-Q_0 u\|$ & order \\
\hline\hline
   Level 1    &1.1606e-01 &            &8.7536e-03 &        \\ \hline
   Level 2    &7.3245e-02 &6.6404e-01  &1.9367e-03 &2.1763  \\ \hline
   Level 3    &4.5864e-02 &6.7538e-01  &4.8418e-04 &2.0000  \\ \hline
   Level 4    &2.8804e-02 &6.7108e-01  &1.8253e-04 &1.4074  \\ \hline
   Level 5    &1.8143e-02 &6.6684e-01  &7.0204e-05 &1.3785  \\ \hline
   Level 6    &1.1453e-02 &6.6372e-01  &2.7002e-05 &1.3785  \\ \hline
\hline
\end{tabular}
\end{table}

\begin{table}[h]
\caption{Example 3. Convergence rate with $k=3$.}\label{ex3_k3}
\center
\begin{tabular}{||c||cccc||}
\hline\hline
Mesh & $|||u_h-Q_h u|||$ & order & $\|u_0-Q_0 u\|$ & order \\
\hline\hline
   Level 1    &3.8466e-02 &            &3.2977e-03 &        \\ \hline
   Level 2    &2.4167e-02 &6.7053e-01  &6.4411e-04 &2.3561  \\ \hline
   Level 3    &1.5215e-02 &6.6757e-01  &1.4313e-04 &2.1699  \\ \hline
   Level 4    &9.5852e-03 &6.6659e-01  &5.3960e-05 &1.4074  \\ \hline
   Level 5    &6.0386e-03 &6.6659e-01  &2.0596e-05 &1.3896  \\ \hline
   Level 6    &3.8042e-03 &6.6662e-01  &7.8013e-06 &1.4005  \\ \hline
\hline
\end{tabular}
\end{table}

\subsection{Example 3}
In this example, we investigate the performance of the HWG method for a problem with a corner singularity. Let $\Omega$ be the L-shaped domain $(-1,1)^2\backslash[0,1)\times(-1,0]$ and impose an appropriate inhomogeneous boundary condition for $u$ so that
$$u=r^{5/3}\sin(5\theta/3),$$
where $(r,\theta)$ denote the system of polar coordinates. In this test, the exact solution $u$ has a singularity at the origin; here, we only have $u\in H^{8/3-\epsilon}(\Omega),$ $\epsilon>0.$
\begin{figure}[h!]
 \begin{center}
\resizebox*{5cm}{5cm}{\includegraphics{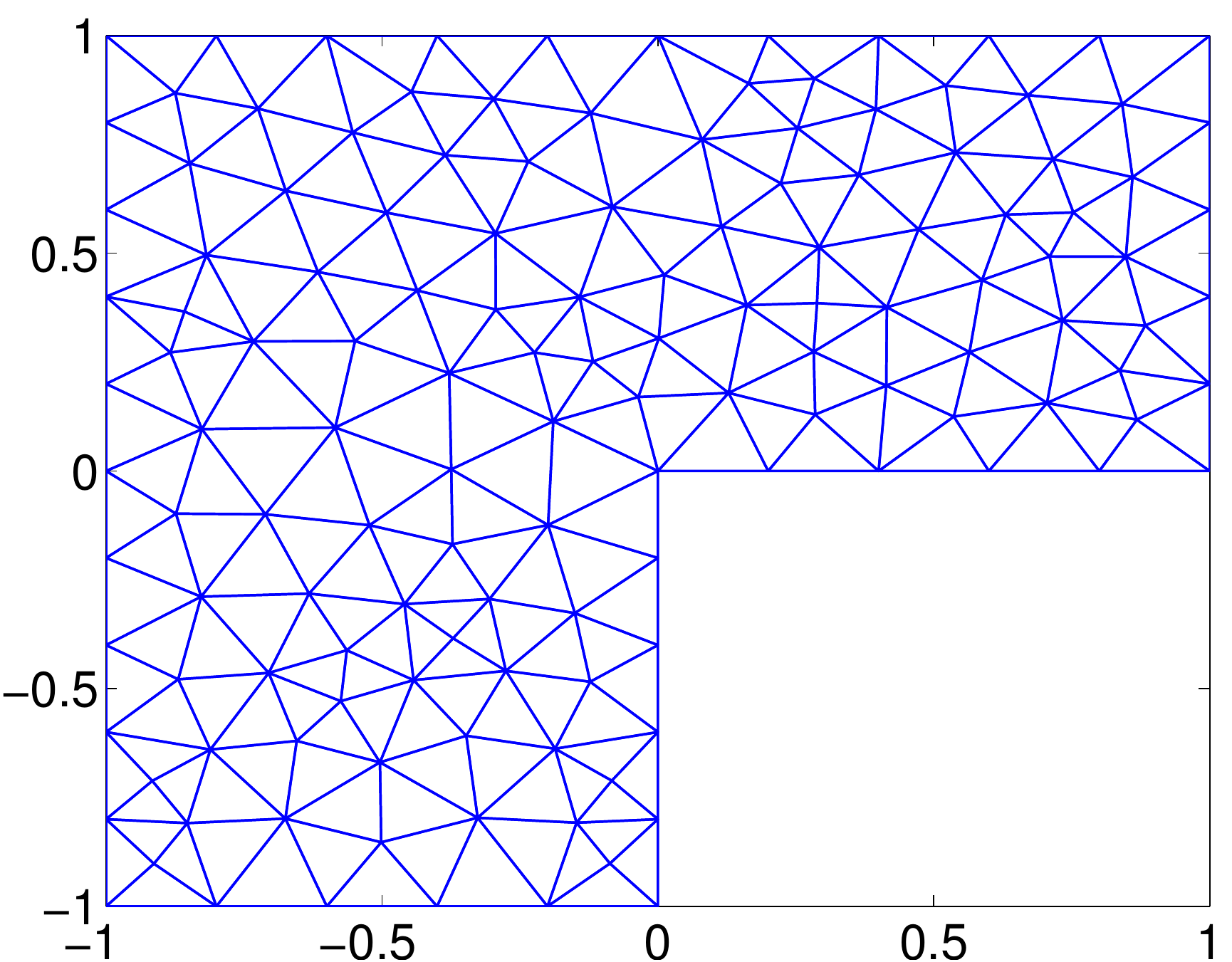}}
\caption{Initial mesh of Example 3}
\label{ex3_fig}
\end{center}
\end{figure}

The initial mesh is shown in Figure \ref{ex3_fig}. The next level of mesh is derived by connecting middle point of each edge for the previous level of mesh. The error of numerical solution is shown in Table \ref{ex3_k2}-\ref{ex3_k3} for $k=2$ and $k=3.$ Here we can observe that $\3bar u_h-Q_h u\3bar$ approaches to zero at the rate $O(h^{2/3})$ as $h\to 0.$ However, the convergence rate of error in $L^2-norm$ is observed as $O(h^{1.4})$.

\end{document}